\documentclass[amssymb,12pt ]{article}
\pdfoutput=1
\usepackage{geometry}
\usepackage{amsfonts}
\usepackage{amssymb}
\usepackage{amsmath}
\usepackage{amsthm}
\usepackage{graphicx}
\usepackage{picture}
\usepackage{curves}
\usepackage{subfigure}
\usepackage{epic}
\usepackage{color}

\newtheorem{thm}{Theorem}[section]

\newtheorem{lem}[thm]{Lemma}
\newtheorem{pro}[thm]{Proposition}

\newenvironment{ack}{\noindent{\bf Acknowledgments}}

\newcommand{\vol}{{\rm vol}}
\newcommand{\cs}{{\rm cs}}
\newcommand{\li}{{\rm Li}_2}

\newcommand{\modulo}{~~({\rm mod}~\pi^2)}

\begin{document}

\title{Optimistic limit of the colored Jones polynomial and the existence of a solution
}
\author{\sc Jinseok Cho}
\maketitle
\begin{abstract}
For the potential function of a link diagram induced by the optimistic limit of the colored Jones polynomial,
we show the existence of a solution of the hyperbolicity equations by directly constructing it.
This construction is based on the shadow-coloring of the conjugation quandle induced by 
a boundary-parabolic representation $\rho:\pi_1(L)\rightarrow{\rm PSL}(2,\mathbb{C})$.
This gives us a very simple and combinatorial method to calculate the complex volume of $\rho$.

\end{abstract}

\section{Introduction}\label{sec1}

The optimistic limit of Kashaev invariant was naturally appeared in \cite{Kashaev97} when the volume conjecture first introduced.
It can be considered as an informal way to predict the actual limit of the Kashaev invariant using a potential function, 
and it has been widely considered the actual limit by general Physicists.
After the appearance, many works have been done to provide a mathematically rigorous definition in 
\cite{Cho09b}, \cite{Yokota10}, \cite{Cho13a} and \cite{Cho14a}.

Let $L$ be a link. 
The author with several collaborators defined a potential function combinatorially from the link diagram in \cite{Cho13a} 
and showed that the evaluation of the function
at a saddle point becomes complex volume of certain representation. Furthermore, it was shown in \cite{Cho14a} that,
if we modify the potential function slightly using the information of 
a given boundary-parabolic representation\footnote{
Boundary-parabolic representation means the images of the meridians and the longitudes of the cusp tori are all parabolic elements
of ${\rm PSL}(2,\mathbb{C})$.} 
$\rho:\pi_1(L)\rightarrow{\rm PSL}(2,\mathbb{C})$,
then the set of hyperbolicity equations always have the solution which induces $\rho$ up to conjugate.
This solution was directly constructed from the shadow-coloring of $\mathcal{P}$ induced by $\rho$,
where $\mathcal{P}$ is the conjugation quandle consists of parabolic elements of ${\rm PSL}(2,\mathbb{C})$.

On the other hand, the colored Jones polynomial was shown to be a generalization of the Kashaev invariant in \cite{Murakami01a},
and the optimistic limit of the colored Jones polynomial was also developed in 
\cite{Thurston99}, \cite{Cho10a}, \cite{Cho13b} and \cite{Cho13c}.
Especially, following the idea of \cite{Cho13a}, another potential function $W(w_1,\ldots,w_n)$ 
from the optimistic limit of the colored Jones polynomial
was defined in \cite{Cho13c} combinatorially from the link diagram. At first, we fix the link diagram\footnote{
We always assume the diagram does not contain a trivial knot component which has only over-crossings or under-crossings or no crossing.
If it happens, then we change the diagram of the trivial component slightly.
For example, applying Reidemeister second move to make different types of crossings 
or Reidemeister first move to add a kink is good enough. 
This assumption is necessary to guarantee that
the five-term triangulation becomes a topological triangulation of $\mathbb{S}^3\backslash(L\cup\{\text{two points}\})$.} $D$. 
Then we assign variables $w_1,\ldots,w_n$ to regions
of the diagram and define a potential function of a crossing $j$ as in Figure \ref{fig01}.

\begin{figure}[h]
\setlength{\unitlength}{0.4cm}
\subfigure[Positive crossing]{
  \begin{picture}(35,6)\thicklines
    \put(6,5){\vector(-1,-1){4}}
    \put(2,5){\line(1,-1){1.8}}
    \put(4.2,2.8){\vector(1,-1){1.8}}
    \put(3.5,1){$w_a$}
    \put(5.5,3){$w_b$}
    \put(3.5,4.5){$w_c$}
    \put(1.5,3){$w_d$}
    \put(8,3){$\longrightarrow$}
    \put(11,4){$W_j:=-\li(\frac{w_c}{w_b})-\li(\frac{w_c}{w_d})+\li(\frac{w_a w_c}{w_b w_d})+\li(\frac{w_b}{w_a})+\li(\frac{w_d}{w_a})$}
    \put(15,2){$-\frac{\pi^2}{6}+\log\frac{w_b}{w_a}\log\frac{w_d}{w_a}$}
    \put(3.6,2){$j$}
  \end{picture}}\\
\subfigure[Negative crossing]{
  \begin{picture}(35,6)\thicklines
    \put(2,5){\vector(1,-1){4}}
    \put(6,5){\line(-1,-1){1.8}}
    \put(3.8,2.8){\vector(-1,-1){1.8}}
    \put(3.5,1){$w_a$}
    \put(5.5,3){$w_b$}
    \put(3.5,4.5){$w_c$}
    \put(1.5,3){$w_d$}
    \put(8,3){$\longrightarrow$}
    \put(11,4){$W_j:=\li(\frac{w_c}{w_b})+\li(\frac{w_c}{w_d})-\li(\frac{w_a w_c}{w_b w_d})-\li(\frac{w_b}{w_a})-\li(\frac{w_d}{w_a})$}
    \put(15,2){$+\frac{\pi^2}{6}-\log\frac{w_b}{w_a}\log\frac{w_d}{w_a}$}
        \put(3.6,2){$j$}
  \end{picture}}
  \caption{Potential function of the crossing $j$}\label{fig01}
\end{figure}

Then the potential function of $D$ is defined by
$$W(w_1,\ldots,w_n):=\sum_{j\text{ : crossings}}W_j,$$
and we modify it to
\begin{equation*}
W_0(w_1,\ldots,w_n):=W(w_1,\ldots,w_n)-\sum_{k=1}^n \left(w_k\frac{\partial W}{\partial w_k}\right)\log w_k.
\end{equation*}

Also, from the potential function $W(w_1,\ldots,w_n)$, we define a set of equations
\begin{equation*}
\mathcal{I}:=\left\{\left.\exp\left(w_k\frac{\partial W}{\partial w_k}\right)=1\right|k=1,\ldots,n\right\}.
\end{equation*}
Then, from Proposition 1.1 of \cite{Cho13c}, $\mathcal{I}$ becomes the set of hyperbolicity equations of
the five-term triangulation of $\mathbb{S}^3\backslash (L\cup\{\text{two points}\})$ {in Figure \ref{fig06}}. Here, hyperbolicity equations
are the equations that determine the complete hyperbolic structure of the triangulation, which consist of
gluing equations of edges and completeness condition. According to Yoshida's construction in Section 4.5 of \cite{Tillmann13},
a solution $\bold w=(w_1,\ldots,w_n)$ of $\mathcal{I}$ determines the boundary-parabolic representation
\begin{equation*}
\rho_{\bold{w}}:\pi_1(\mathbb{S}^3\backslash (L\cup\{\text{two points}\}))=\pi_1(\mathbb{S}^3\backslash L)\longrightarrow{\rm PSL}(2,\mathbb{C}).
\end{equation*} 

Theorem 1.2 of \cite{Cho13c} shows that, for the solution $\bold w$ of $\mathcal{I}$,
\begin{equation}\label{W1}
W_0(\bold{w})\equiv i(\vol(\rho_{\bold w})+i\,\cs(\rho_{\bold w}))\modulo,
\end{equation}
where $\vol(\rho_{\bold w})$ and $\cs(\rho_{\bold w})$ are the volume and the Chern-Simons invariant of the representation
$\rho_{\bold w}$ defined in \cite{Zickert09}, respectively. 
We call $\vol(\rho_{\bold w})+i\,\cs(\rho_{\bold w})$ {\it the complex volume of $\rho_{\bold w}$}.

Although the potential function in \cite{Cho13c} determines the complex volume very nicely, there are two major problems. 
\begin{enumerate}
  \item When $\mathcal{I}$ has no solution, we cannot do anything with the potential function $W$.
  \item We do not know whether the set $\{\rho_{\bold w}\,\vert\, {\bold w}\text{ is a solution of }\mathcal{I}\}$ contains 
     all possible boundary-parabolic representations $\rho:\pi_1(L)\rightarrow{\rm PSL}(2,\mathbb{C})$.
\end{enumerate}

In the case of the optimistic limit of Kashaev invariant, we solved these problems in \cite{Cho14a} by using the shadow-coloring of 
the conjugation quandle $\mathcal{P}$ defined in \cite{Kabaya14}.
The purpose of this article is to solve the above problems 
by constructing a solution of $\mathcal{I}$ using the same method.

\begin{thm}\label{thm1}
For any boundary-parabolic representation $\rho:\pi_1(L)\rightarrow{\rm PSL}(2,\mathbb{C})$ and any link diagram $D$ of $L$,
there exists the solution $\bold w^{(0)}$ of $\mathcal{I}$ satisfying
$\rho_{\bold w^{(0)}}=\rho,$
up to conjugate.
\end{thm}

The exact formula of $\bold w^{(0)}$ is in (\ref{main}), which is amazingly simple.
Using this solution, we define 
{\it the colored Jones version of the optimistic limit of} $\rho$ by $W_0({\bold w^{(0)}})$.
Then, from (\ref{W1}), the optimistic limit is always the complex volume of $\rho$.
The author believes calculating this optimistic limit is the most convenient method to obtain
the complex volume of a given boundary-parabolic representation
because everything is combinatorially obtained from the link diagram.

Note that, the potential function and the triangulation of the Kashaev version in \cite{Cho13a} was slightly modified in \cite{Cho14a}
according to the information of the representation $\rho$ so as to guarantee the existence of a solution. 
However, we do not need any modification of the colored Jones version in \cite{Cho13c}, which is a major advantage of this article.
{ Actually, if a link diagram contains Figure \ref{nosol} or a kink, then the unmodified Kashaev version does not have
solutions. (See \cite{Cho13c} for the proof. The modification needs extra information other than the link diagram.)
Due to the existence of a solution of the colored Jones version for any diagram and any $\rho$, 
several combinatorial applications are possible. 
See \cite{Cho15b} and \cite{Cho15a} for those applications.

\begin{figure}[h]
\begin{center}\includegraphics[scale=1]{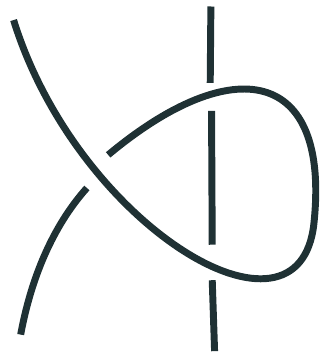}\end{center}
\caption{Example that Kashaev version does not have a solution}\label{nosol}
\end{figure}}

\section{Construction of the solution}
\subsection{Reviews on shadow-coloring}

This section is a summary of definitions and properties we need. 
For complete descriptions, see {\cite{Kabaya14}, especially Section 5}.

Let $\mathcal{P}$ be the set of parabolic elements of ${\rm PSL}(2,\mathbb{C})={\rm Isom^+}(\mathbb{H}^3)$.
We identify $\mathbb{C}^2\backslash\{0\}/\pm$ with $\mathcal{P}$ by
$$\left(\begin{array}{c}\alpha \\\beta\end{array}\right)  
  \longleftrightarrow\left(\begin{array}{cc}1+\alpha\beta & -\alpha^2 \\ \beta^2& 1-\alpha\beta\end{array}\right),$$
and define operation $*$ by
\begin{eqnarray*}
  \left(\begin{array}{c}\alpha \\\beta\end{array}\right)*\left(\begin{array}{c}\gamma \\ \delta\end{array}\right)
  =\left(\begin{array}{cc}1+\gamma\delta & -\gamma^2 \\ \delta^2& 1-\gamma\delta\end{array}\right)
  \left(\begin{array}{c}\alpha \\\beta\end{array}\right)\in\mathcal{P},
\end{eqnarray*}
where this operation is actually induced by the conjugation as follows:
$$  \left(\begin{array}{c}\alpha \\\beta\end{array}\right)*\left(\begin{array}{c}\gamma \\ \delta\end{array}\right)\in\mathcal{P}
\longleftrightarrow \left(\begin{array}{c}\gamma \\\delta\end{array}\right)
  \left(\begin{array}{c}\alpha \\ \beta\end{array}\right)
  \left(\begin{array}{c}\gamma \\\delta\end{array}\right)^{-1}\in{\rm PSL}(2,\mathbb{C}).$$
The inverse operation $*^{-1}$ is expressed by
$$\left(\begin{array}{c}\alpha \\\beta\end{array}\right)*^{-1}\left(\begin{array}{c}\gamma \\ \delta\end{array}\right)
  =\left(\begin{array}{cc}1-\gamma\delta & \gamma^2 \\ -\delta^2& 1+\gamma\delta\end{array}\right)
  \left(\begin{array}{c}\alpha \\\beta\end{array}\right)\in\mathcal{P},$$
and $(\mathcal{P},*)$ becomes a conjugation quandle. 
Here, quandle means, for any $a,b,c\in\mathcal{P}$, the map $*b:a\mapsto a*b$ is bijective and
$$a*a=a, ~(a*b)*c=(a*c)*(b*c)$$
hold.

We define {\it the Hopf map} $h:\mathcal{P}\rightarrow\mathbb{CP}^1=\mathbb{C}\cup\{\infty\}$ by
$$\left(\begin{array}{c}\alpha \\\beta\end{array}\right)\mapsto \frac{\alpha}{\beta}.$$

For an oriented link diagram $D$ of $L$ and the boundary-parabolic representation $\rho$, we assign {\it arc-colors}
$a_1,\ldots,a_n\in\mathcal{P}$ to arcs of $D$ so that each $a_k$ is the image of the meridian around the arc
under the representation $\rho$. Note that, in Figure \ref{fig02}, we have 
\begin{equation}\label{ope}
     a_m=a_l*a_k.
\end{equation}

\begin{figure}[h]
\centering  \setlength{\unitlength}{0.6cm}\thicklines
\begin{picture}(4.5,5.5)(1.5,0.5)  
    \put(6,5){\vector(-1,-1){4}}
    \put(2,5){\line(1,-1){1.8}}
    \put(4.2,2.8){\line(1,-1){1.8}}
    \put(6.2,5.2){$a_k$}
    \put(1.5,0.5){$a_k$}
    \put(1.3,5.2){$a_l$}
    \put(6.2,0.5){$a_m$}
  \end{picture}
  \caption{Arc-coloring}\label{fig02}
\end{figure}
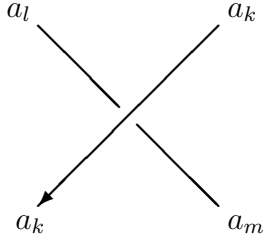

We also assign {\it region-colors} $s_1,\ldots,s_m\in\mathcal{P}$ to regions of $D$ satisfying the rule in Figure \ref{fig03}.
Note that, if an arc-coloring is given, then a choice of one region-color determines all the other region-colors.

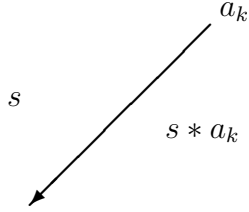
\begin{figure}[h]
\centering  \setlength{\unitlength}{0.6cm}\thicklines
\begin{picture}(6,5)  
    \put(6,4){\vector(-1,-1){4}}
    \put(1.5,2.2){$s$}
    \put(5,1.5){$s*a_k$}
    \put(6.2,4.2){$a_k$}
  \end{picture}
  \caption{Region-coloring}\label{fig03}
\end{figure}

\begin{lem}\label{lem} 
Consider the arc-coloring induced by the boundary-parabolic representation $\rho:\pi_1(L)\rightarrow {\rm PSL}(2,\mathbb{C})$.
Then, for any triple $(a_k,s,s*a_k)$ of an arc-color $a_k$ and its surrounding region-colors $s, s*a_k$ as in Figure \ref{fig03},
there exists a region-coloring satisfying
\begin{equation*}
  h(a_k)\neq h(s)\neq h(s*a_k)\neq h(a_k).
\end{equation*}
\end{lem}

\begin{proof} Although it was already proved in Proposition 2 of \cite{Kabaya14} and Lemma 2.4 of \cite{Cho14a}, 
we write down the proof again for the reader's convenience.

For the given arc-colors $a_1,\ldots,a_n$, we choose region-colors $s_1,\ldots,s_m$ so that
\begin{equation}\label{exi}
  \{h(s_1),\ldots,h(s_m)\}\cap\{h(a_1),\ldots,h(a_n)\}=\emptyset.
\end{equation}
This is always possible because, the number of $h(s_1)$ satisfying $h(s_1)\in\{h(a_1),\ldots,h(a_n)\}$ is finite, and $h(s_2),\ldots,h(s_m)$
are uniquely determined by $h(s_1)$. Therefore, the number of $h(s_1)$ satisfying
\begin{equation*}
  \{h(s_1),\ldots,h(s_m)\}\cap\{h(a_1),\ldots,h(a_n)\}\neq\emptyset
\end{equation*}
 is finite, but we have infinitely many choice of the value $h(s_1)\in\mathbb{CP}^1$.
 
Now consider Figure \ref{fig03} and assume
$h(s*a_k)=h(s)$. Then we obtain
\begin{equation}\label{eqnh}
  h(s*a_k)=\widehat{a_k}(h(s))=h(s),
\end{equation}
where $\widehat{a_k}:\mathbb{CP}^1\rightarrow\mathbb{CP}^1$ is the M\"{o}bius transformation
\begin{equation}\label{mob}
\widehat{a_k}(z)=\frac{(1+\alpha_k\beta_k)z-\alpha_k^2}{\beta_k^2 z+(1-\alpha_k\beta_k)}
\end{equation}
of $a_k=\left(\begin{array}{c}\alpha_k \\\beta_k\end{array}\right)$. Then (\ref{eqnh}) implies
$h(s)$ is the fixed point of $\widehat{a_k}$, which means $h(a_k)=h(s)$ that contradicts (\ref{exi}).

\end{proof}

{We remark that Lemma \ref{lem} holds for any choice of region-colors with only finitely many exceptions. 
Therefore, if we want to find a region-color explicitly, 
we choose $h(s_1)\notin\{h(a_1),\ldots,h(a_n)\}$ and then decide $h(s_2), \ldots, h(s_m)$ using 
$$h(s_1*a)=\widehat{a}(h(s_1)),~h(s_1*^{-1} a)=\widehat{a}^{-1}(h(s_1)).$$
If this choice does not satisfy Lemma \ref{lem}, then we change $h(s_1)$ and do the same process again.
This process is very simple and it ends in finite steps. If proper $h(s_1)$ is chosen, then we can easily 
extend it to $s_1\in\mathcal{P}$ and decide proper region-coloring $\{s_1,\ldots,s_m\}$.}

The arc-coloring induced by $\rho$ together with the region-coloring satisfying Lemma \ref{lem}
is called {\it the shadow-coloring induced by} $\rho$. 
We choose $p\in\mathcal{P}$ so that 
\begin{equation}\label{p}
h(p)\notin\{h(a_1),\ldots,h(a_n), h(s_1),\ldots,h(s_m)\}.
\end{equation}
The geometric shape of the five-term triangulation will be determined by the shadow-coloring induced by $\rho$ and $p$ 
in the following section.

From now on, we fix the representatives of shadow-colors in $\mathbb{C}^2\backslash\{0\}$.
As mentioned in \cite{Cho14a}, the representatives of some arc-colors may satisfy (\ref{ope}) up to sign, 
in other words, $a_m=\pm (a_l*a_k)$ {in Figure \ref{fig02}}. 
However, the representatives of the region-colors are uniquely determined due to the fact $s*(\pm a)=s*a$
for any region-color $s$ and any arc-color $a$.

For $a=\left(\begin{array}{c}\alpha_1 \\\alpha_2\end{array}\right)$ and $
b=\left(\begin{array}{c}\beta_1 \\\beta_2\end{array}\right)$ in $\mathbb{C}^2\backslash\{0\}$,
we define {\it determinant} $\det(a,b)$ by
\begin{equation*}
  \det(a,b):=\det\left(\begin{array}{cc}\alpha_1 & \beta_1 \\\alpha_2 & \beta_2\end{array}\right)=\alpha_1 \beta_2-\beta_1 \alpha_2.
\end{equation*}
Then it satisfies
$\det(a*c,b*c)=\det(a,b)$
for any $a,b,c\in\mathbb{C}^2\backslash\{0\}$. Furthermore, for $v_0,\ldots,v_3\in\mathbb{C}^2\backslash \{0\}$,
the cross-ratio can be expressed using the determinant by
$$[h(v_0),h(v_1),h(v_2),h(v_3)]=\frac{\det(v_0,v_3)\det(v_1,v_2)}{\det(v_0,v_2)\det(v_1,v_3)}.$$
(For the proof of these, see {Lemma 2.9} of \cite{Cho14a}.)

\subsection{Geometric shape of the five-term triangulation}\label{sec22}

The five-term triangulation is obtained by placing octahedra on each crossings and subdivide them into five tetrahedra.
(See Section 3 of \cite{Cho13c} for exact description.)

Consider the crossing in Figure \ref{fig04} with the shadow-coloring induced by $\rho$, and let $w_a,\ldots,w_d$
be the variables assigned to regions of $D$. 

\begin{figure}[h]
\centering
\subfigure[Positive crossing]{\begin{picture}(6,5.5)  
  \setlength{\unitlength}{0.8cm}\thicklines
    \put(6,5){\vector(-1,-1){4}}
    \put(2,5){\line(1,-1){1.8}}
    \put(4.2,2.8){\vector(1,-1){1.8}}
    \put(6.2,5.2){$a_k$}
    \put(1.3,0.8){$a_k$}
    \put(1.3,5.2){$a_l$}
    \put(6.2,0.8){$a_l*a_k$}
    \put(3.5,4.7){$s*a_l$}
    \put(3.9,4.2){$w_c$}
    \put(2,3){$s$}
    \put(1.9,2.5){$w_d$}
    \put(5.5,3){\small $(s*a_l)*a_k$}
    \put(6.5,2.5){$w_b$}
    \put(3.5,1){$s*a_k$}
    \put(3.9,0.5){$w_a$}
  \end{picture}}
\subfigure[Negative crossing]{\begin{picture}(6,5.5)  
  \setlength{\unitlength}{0.8cm}\thicklines
    \put(6,5){\vector(-1,-1){4}}
    \put(3.8,3.2){\vector(-1,1){1.8}}
    \put(4.2,2.8){\line(1,-1){1.8}}
    \put(6.2,5.2){$a_k$}
    \put(1.3,0.8){$a_k$}
    \put(1.3,5.2){$a_l$}
    \put(6.2,0.8){$a_l*a_k$}
    \put(4,4.7){$s$}
    \put(3.9,4.2){$w_d$}
    \put(1.2,3){$s*a_l$}
    \put(1.5,2.5){$w_a$}
    \put(5.5,3){$s*a_k$}
    \put(5.8,2.5){$w_c$}
    \put(3,1){\small $(s*a_l)*a_k$}
    \put(3.9,0.5){$w_b$}
  \end{picture}}
 \caption{Crossing with shadow-coloring and region variables}\label{fig04}
\end{figure}
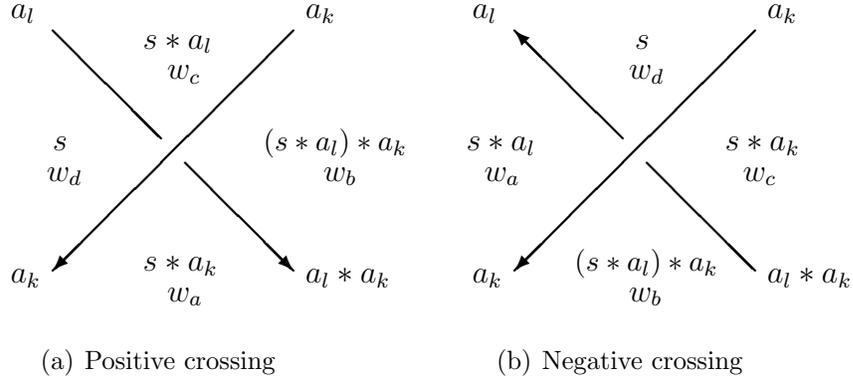

We place tetrahedra at each crossings of $D$ and assign coordinates as in Figure \ref{fig05}
so as to make them hyperbolic ideal ones.
As a matter of fact, Figure \ref{fig05} is the same with Figure 11 of \cite{Cho14a} without orientations, 
which was used only for a degenerate crossing.\footnote{
A crossing is called {\it degenerate} when $h(a_k)=h(a_l)$ holds 
for the two arcs of the crossing with arc-colors $a_k$ and $a_l$.}
Interestingly, this subdivision is good enough for our purpose whether the crossing is degenerate or not.

\begin{figure}[h]
\centering
\subfigure[Positive crossing]{\includegraphics[scale=0.8]{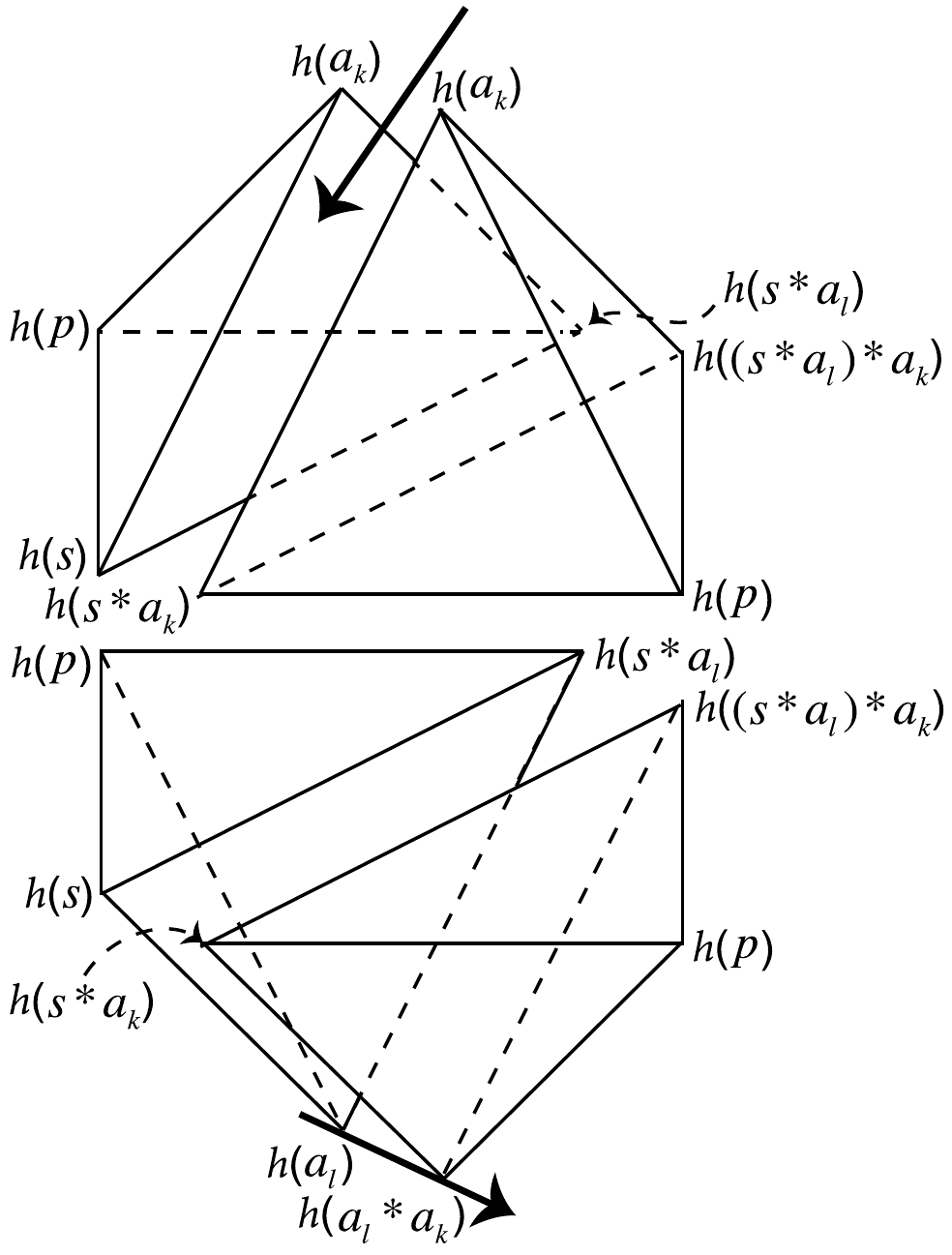}}
\subfigure[Negative crossing]{\includegraphics[scale=0.8]{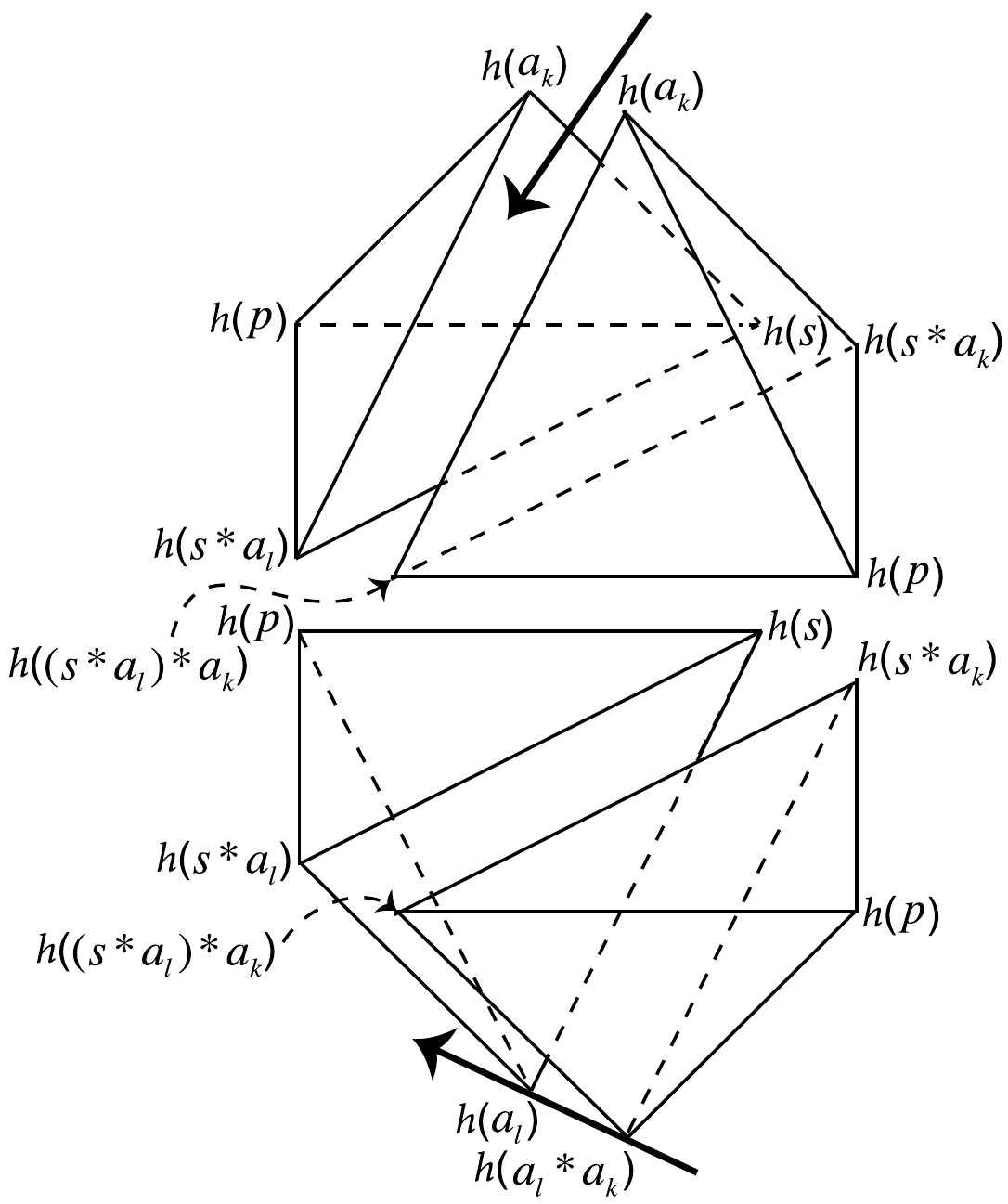}}
 \caption{Cooridnates of tetrahedra at the crossing in Figure \ref{fig04}}\label{fig05}
\end{figure}

\begin{lem}\label{lem2}
All the tetrahedra in Figure \ref{fig05} are non-degenerate.
\end{lem}

\begin{proof} 
It is trivial because the shadow-coloring we are considering satisfies Lemma \ref{lem},
and all endpoints of edges are adjacent, as $h(a_k), h(s), h(s*a_k)$ in Figure \ref{fig03}, or one of them is $h(p)$.

\end{proof}

According to Yoshida's construction in Section 4.5 of \cite{Tillmann13},
the shape of the triangulation according to the coordinates in Figure \ref{fig05} determines a boundary-parabolic representation
$\rho':\pi_1(L)\rightarrow{\rm PSL}(2,\mathbb{C})$. However, $\rho'$ equals to $\rho$ up to conjugate because,
due to the Poincar\'{e} polyhedron theorem, $\pi_1(L)$ is generated by face-pairing maps. In Figure \ref{fig05},
the face-pairing maps are the isomorphisms induced by M\"{o}bius transformations of ${a_1},\ldots,a_n\in{\rm PSL}(2,\mathbb{C})$. 
Therefore, two representations $\rho$ and $\rho'$
send generators to the same elements $a_1,\ldots,a_n$, which means $\rho=\rho'$ up to conjugate.

To make the five-term triangulation, we glue the face $\{h(a_k), h(s), h(s*a_l)\}$ to $\{h(a_k), h(s*a_k), h((s*a_l)*a_k)\}$
by sending the tetrahedron $\{h(a_k), h(s), h(s*a_l), h(p)\}$ by the isomorphism induced by ${a}_k$.
After applying 2-3 move along the edge $\{h(p*a_k),h(p)\}$, we obtain Figure \ref{fig06}.
Here we assigned the vertex-orientation according to Figure 9 of \cite{Cho13c}.

\begin{figure}[h]
\centering
\subfigure[Positive crossing]{\includegraphics[scale=0.8]{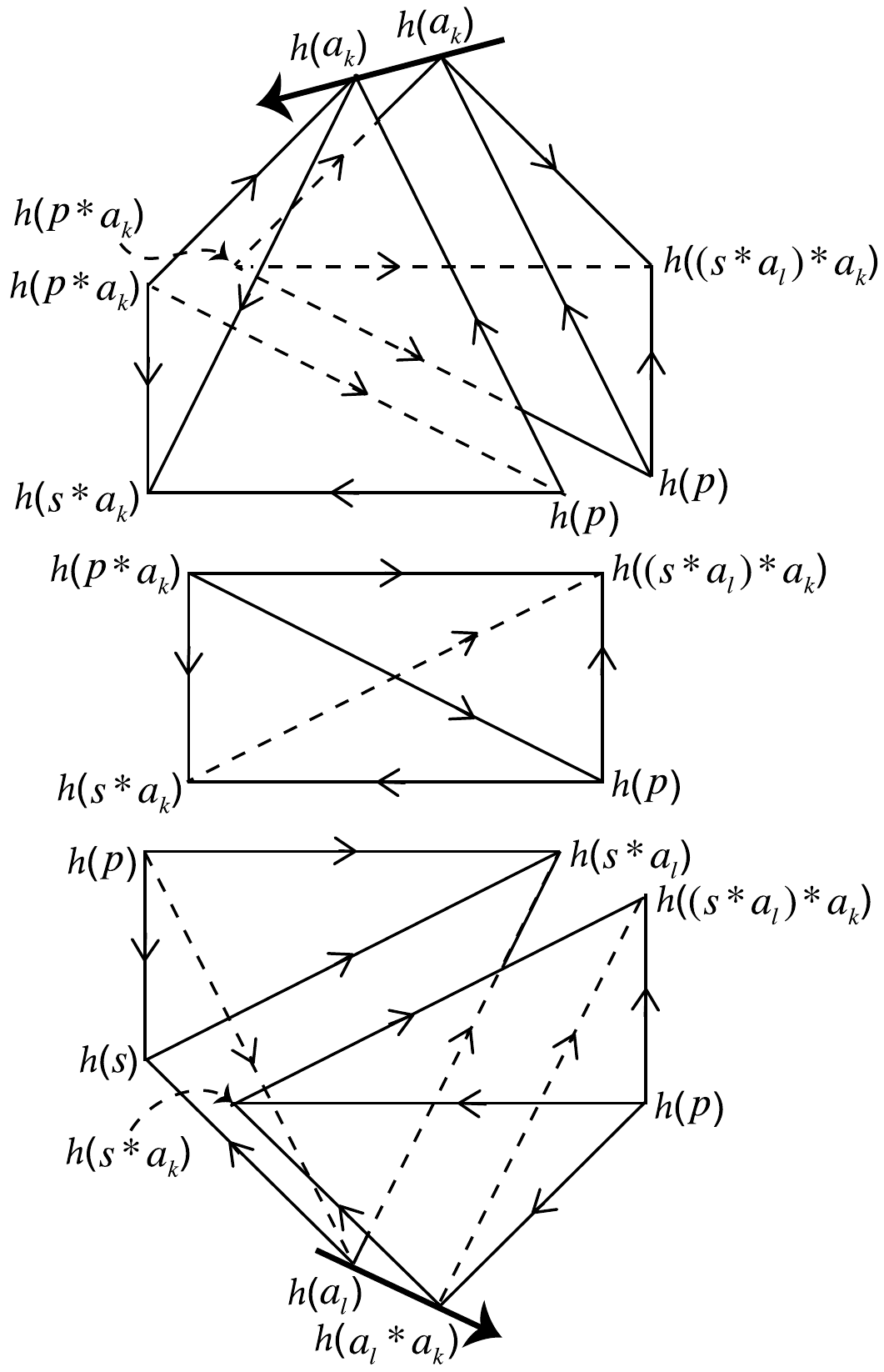}}
\subfigure[Negative crossing]{\includegraphics[scale=0.8]{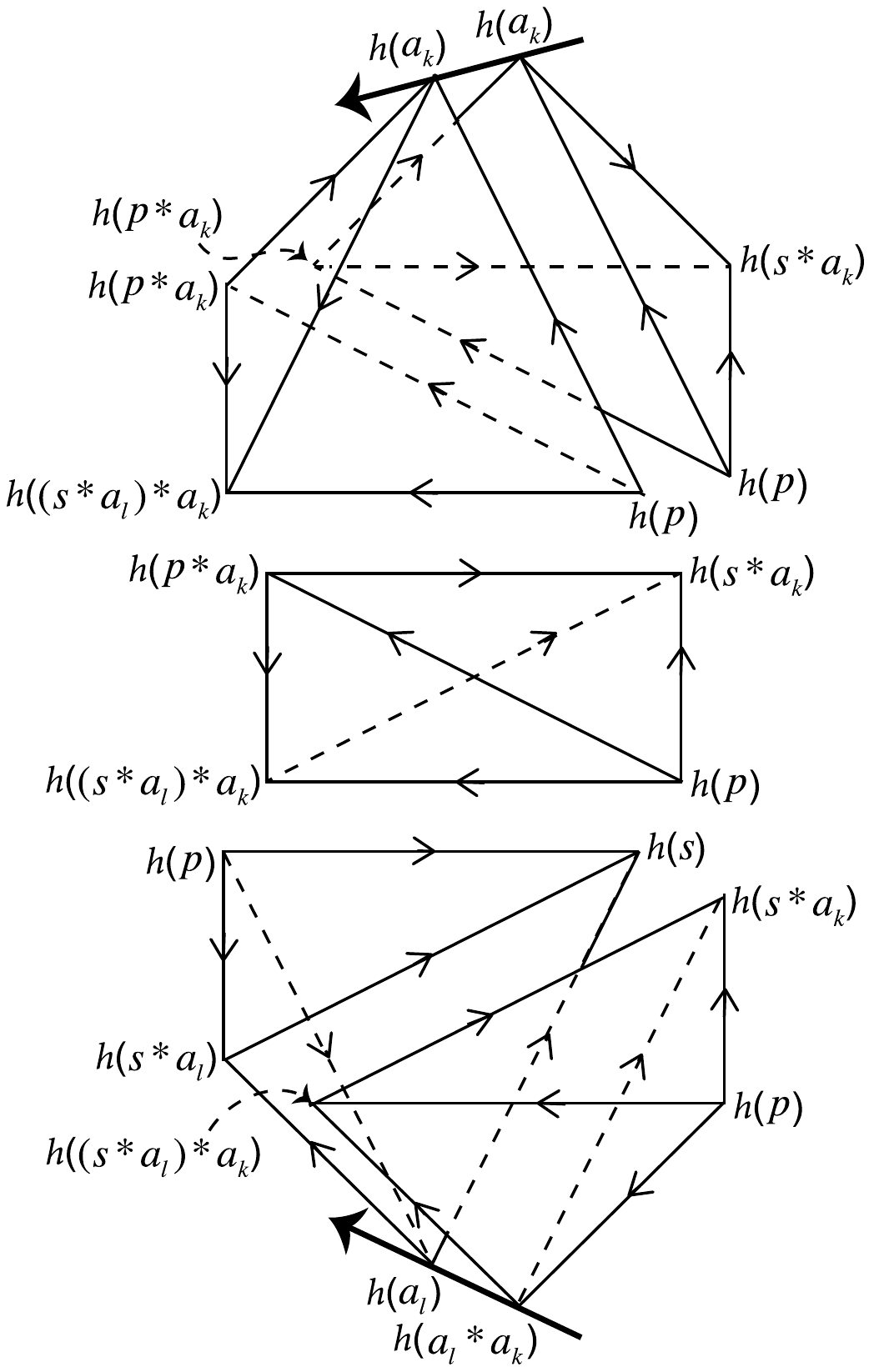}}
 \caption{Five-term triangulation at the crossing in Figure \ref{fig04}}\label{fig06}
\end{figure}

\begin{pro}\label{pro} All the tetrahedra in Figure \ref{fig06} are non-degenerate.
\end{pro}

\begin{proof} All the edges of the tetrahedra were already appeared in Lemm \ref{lem2} expect $\{h(p*a_k),h(p)\}$,
so it is enough to show that $h(p*a_k)\neq h(a_k)$.

Assume $h(p*a_k)=h(a_k)$. Then 
\begin{equation}\label{eqnh}
  h(p*a_k)=\widehat{a_k}(h(p))=h(p),
\end{equation}
where $\widehat{a_k}:\mathbb{CP}^1\rightarrow\mathbb{CP}^1$ is the M\"{o}bius transformation
of $a_k$ defined in (\ref{mob}). Then (\ref{eqnh}) implies
$h(p)$ is the fixed point of $\widehat{a_k}$, which means $h(a_k)=h(p)$. This contradicts the definition (\ref{p}) of $p$.

\end{proof}

\subsection{Formula of the solution ${\bold w}^{(0)}$}\label{sec23}
Consider the crossing in Figure \ref{fig04} and the tetrahedra in Figure \ref{fig06}. For the positive crossing,
we assign shape parameters to the edges as follows:
\begin{itemize} 
\item $\displaystyle\frac{w_d}{w_a}$ to $(h(a_k), h(s*a_k))$ of $(h(p*a_k),h(p),h(a_k), h(s*a_k))$, 
\item $\displaystyle\frac{w_b}{w_c}$ to $(h(a_k), h((s*a_l)*a_k))$ of $-(h(p*a_k),h(p),h(a_k), h((s*a_l)*a_k))$, 
\item $\displaystyle\frac{w_b}{w_a}$ to $(h(p), h(a_l*a_k))$ of $(h(p), h(a_l*a_k),h(s*a_k), h((s*a_l)*a_k))$ and 
\item $\displaystyle\frac{w_d}{w_c}$ to $(h(p), h(a_l))$ of $-(h(p), h(a_l),h(s), h(s*a_l))$, respectively.
\end{itemize}
On the other hand, for the negative crossing,
we assign shape parameters to the edges as follows:
\begin{itemize} 
\item $\displaystyle\frac{w_a}{w_b}$ to $(h(a_k), h((s*a_l)*a_k))$ of $-(h(p),h(p*a_k),h(a_k), h((s*a_l)*a_k))$, 
\item $\displaystyle\frac{w_c}{w_d}$ to $(h(a_k), h(s*a_k))$ of $(h(p),h(p*a_k),h(a_k), h(s*a_k))$, 
\item $\displaystyle\frac{w_c}{w_b}$ to $(h(p), h(a_l*a_k))$ of $(h(p), h(a_l*a_k), h((s*a_l)*a_k),h(s*a_k))$ and 
\item $\displaystyle\frac{w_a}{w_d}$ to $(h(p), h(a_l))$ of $-(h(p), h(a_l),h(s*a_l), h(s))$, respectively.
\end{itemize}
According to Proposition 1.1 of \cite{Cho13c}, $\mathcal{I}$ becomes the set of hyperbolicity equations 
of the five-term triangulation with the above shape parameters.

For a region of $D$ with region-color $s_k$ and region-variable $w_k$, we define
\begin{equation}\label{main}
w_k^{(0)}:=\det(p,s_k).
\end{equation}
Then, by the definition of $p$, we know $w_k^{(0)}\not= 0$. Furthermore, direct calculation shows 
${\bold w^{(0)}}=(w_1^{(0)},\ldots,w_n^{(0)})$ is a solution of $\mathcal{I}$. 
Specifically, for the first two cases of the positive crossing, the shape parameters
assigned to edges $(h(a_k), h(s*a_k))$ and $(h(a_k), h((s*a_l)*a_k))$ are the cross-ratios
\begin{eqnarray*}
[h(p*a_k),h(p),h(a_k), h(s*a_k)]&=&\frac{\det(p*a_k,s*a_k)\det(p,a_k)}{\det(p*a_k,a_k)\det(p,s*a_k)}\\
&=&\frac{\det(p,s)\det(p,a_k)}{\det(p,a_k)\det(p,s*a_k)}=\frac{w_d^{(0)}}{w_a^{(0)}},\\
{[h}(p*a_k),h(p),h(a_k), h((s*a_l)*a_k)]^{-1}
&=&\frac{\det(p*a_k,a_k)\det(p,(s*a_l)*a_k)}{\det(p*a_k,(s*a_l)*a_k)\det(p,a_k)}\\
&=&\frac{\det(p,a_k)\det(p,(s*a_l)*a_k)}{\det(p,s*a_l)\det(p,a_k)}=\frac{w_b^{(0)}}{w_c^{(0)}},
\end{eqnarray*}
respectively, and all the other cases can be verified by the same way. 
The proof of Theorem \ref{thm1} follows from the above and the discussion below Lemma \ref{lem2}.
 
\section{Examples}

We consider the same examples in Section 4 of \cite{Cho14a}.
\subsection{Figure-eight knot $4_1$}\label{sec31}

\begin{figure}[h]\centering
\includegraphics[scale=0.5]{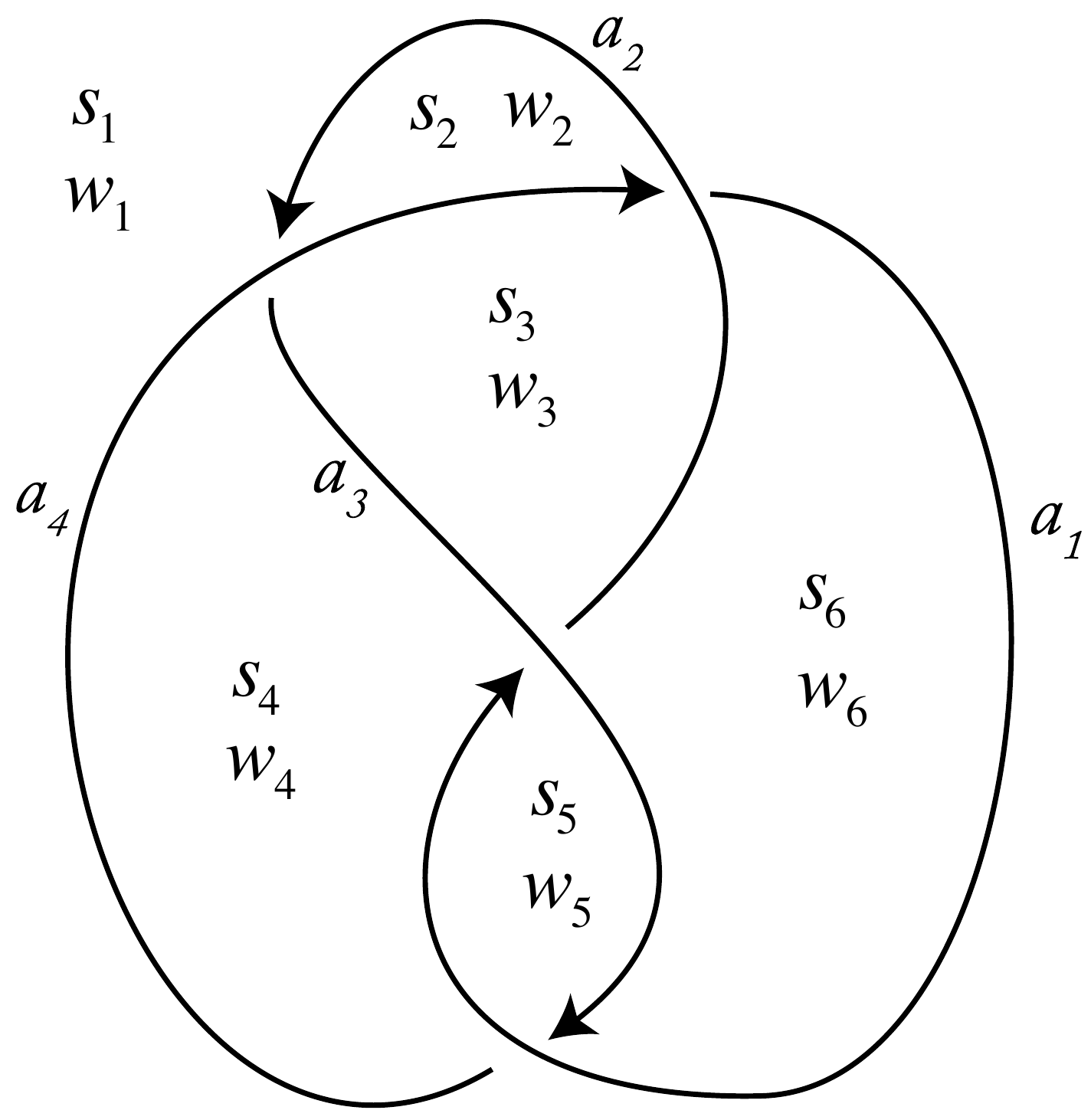}
\caption{Figure-eight knot $4_1$ with parameters}\label{example1}
\end{figure}

For the figure-eight knot diagram in Figure \ref{example1}, 
let the representatives of the shadow-coloring be 
\begin{align*}
a_1=\left(\begin{array}{c}0 \\t\end{array}\right),~a_2=\left(\begin{array}{c}1 \\0\end{array}\right),~
a_3=\left(\begin{array}{c}-t \\1+t\end{array}\right),~a_4=\left(\begin{array}{c}-t \\t\end{array}\right),\\
s_1=\left(\begin{array}{c}1 \\1\end{array}\right),~s_2=\left(\begin{array}{c}0 \\1\end{array}\right),~
s_3=\left(\begin{array}{c}-t-1 \\t+2\end{array}\right),~s_4=\left(\begin{array}{c}-2t-1 \\2t+3\end{array}\right),\\
s_5=\left(\begin{array}{c}-2t-1 \\t+4\end{array}\right),~s_6=\left(\begin{array}{c}1 \\t+2\end{array}\right),~
p=\left(\begin{array}{c}2 \\1\end{array}\right),
\end{align*}
where $t$ is a solution of $t^2+t+1=0$, and
let $\rho:\pi_1(4_1)\rightarrow{\rm PSL}(2,\mathbb{C})$ be the boundary-parabolic representation determined by $a_1,\ldots,a_4$.

The potential function $W(w_1,\ldots,w_6)$ of Figure \ref{example1} is
\begin{align*}
W=\left\{\li(\frac{w_1}{w_2})+\li(\frac{w_1}{w_4})-\li(\frac{w_1 w_3}{w_2 w_4})
-\li(\frac{w_2}{w_3})-\li(\frac{w_4}{w_3})+\frac{\pi^2}{6}-\log\frac{w_2}{w_3}\log\frac{w_4}{w_3}\right\}\\
+\left\{\li(\frac{w_3}{w_2})+\li(\frac{w_3}{w_6})-\li(\frac{w_1 w_3}{w_2 w_6})
-\li(\frac{w_2}{w_1})-\li(\frac{w_6}{w_1})+\frac{\pi^2}{6}-\log\frac{w_2}{w_1}\log\frac{w_6}{w_1}\right\}\\
+\left\{-\li(\frac{w_4}{w_3})-\li(\frac{w_4}{w_5})+\li(\frac{w_4 w_6}{w_3 w_5})
+\li(\frac{w_3}{w_6})+\li(\frac{w_5}{w_6})-\frac{\pi^2}{6}+\log\frac{w_3}{w_6}\log\frac{w_5}{w_6}\right\}\\
+\left\{-\li(\frac{w_6}{w_1})-\li(\frac{w_6}{w_5})+\li(\frac{w_4 w_6}{w_1 w_5})
+\li(\frac{w_1}{w_4})+\li(\frac{w_5}{w_4})-\frac{\pi^2}{6}+\log\frac{w_1}{w_4}\log\frac{w_5}{w_4}\right\}.
\end{align*} 

Applying (\ref{main}), we obtain
\begin{align*}
w_1^{(0)}=\det(p,s_1)=1,~w_2^{(0)}=\det(p,s_2)=2,~w_3^{(0)}=\det(p,s_3)=3t+5,\\
w_4^{(0)}=\det(p,s_4)=6t+7,~w_5^{(0)}=\det(p,s_5)=4t+9,~w_6^{(0)}=\det(p,s_6)=2t+3,
\end{align*}
and $(w_1^{(0)},\ldots,w_6^{(0)})$ becomes a solution of 
$\mathcal{I}=\{\exp(w_k\frac{\partial W}{\partial w_k})=1~|~k=1,\ldots,6\}$.
Applying (\ref{W1}), we obtain
\begin{equation*}
W_0(w_1^{(0)},\ldots,w_6^{(0)})\equiv i(\vol(\rho)+i\,\cs(\rho))\modulo,
\end{equation*}
and numerical calculation verifies it by
\begin{equation*}
W_0(w_1^{(0)},\ldots,w_6^{(0)})=
      \left\{\begin{array}{ll}i(2.0299...+0\,i)=i(\vol(4_1)+i\,\cs(4_1))&\text{ if }t=\frac{-1-\sqrt{3} \,i}{2}, \\
                i(-2.0299...+0\,i)=i(-\vol(4_1)+i\,\cs(4_1))&\text{ if }t=\frac{-1+\sqrt{3}\,i}{2}. \end{array}\right.
\end{equation*}

\subsection{Trefoil knot $3_1$}\label{sec32}

\begin{figure}[h]\centering
\includegraphics[scale=0.6]{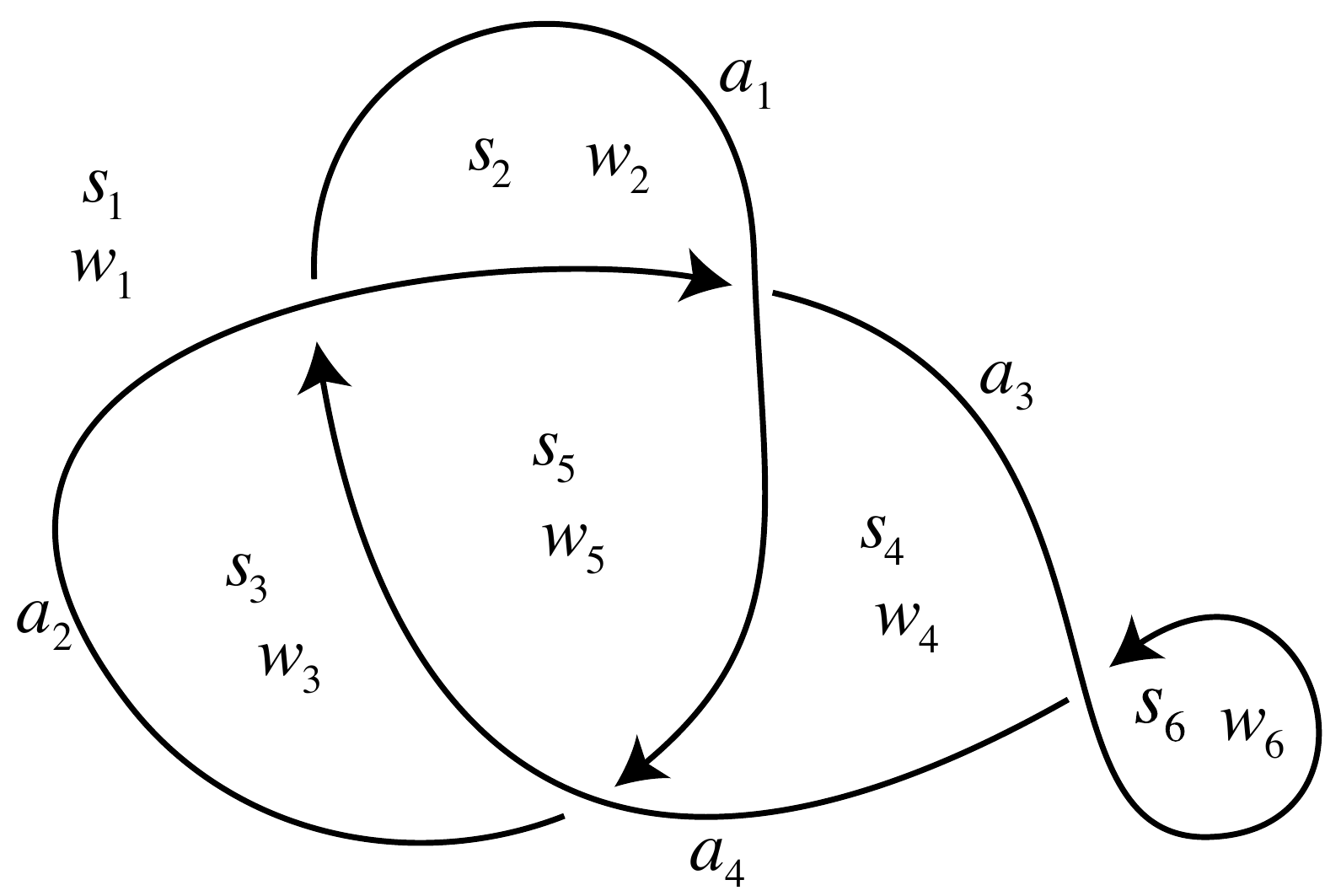}
\caption{Trefoil knot $3_1$ with parameters}\label{example2}
\end{figure}
For the trefoil knot diagram in Figure \ref{example2}, 
let the representatives of the shadow-coloring be 
\begin{align*}
a_1=\left(\begin{array}{c}1 \\0\end{array}\right),~a_2=\left(\begin{array}{c}0 \\1\end{array}\right),~
a_3=a_4=\left(\begin{array}{c}-1 \\1\end{array}\right),\\
s_1=\left(\begin{array}{c}-1 \\2\end{array}\right),~s_2=\left(\begin{array}{c}1 \\2\end{array}\right),~
s_3=\left(\begin{array}{c}-1 \\3\end{array}\right),~s_4=\left(\begin{array}{c}0 \\1\end{array}\right),\\
s_5=\left(\begin{array}{c}1 \\1\end{array}\right),s_6=\left(\begin{array}{c}-2 \\3\end{array}\right),
~p=\left(\begin{array}{c}2 \\1\end{array}\right),
\end{align*}
and let $\rho:\pi_1(3_1)\rightarrow{\rm PSL}(2,\mathbb{C})$
be the boundary-parabolic representation determined by $a_1,a_2,a_3, a_4$.

The potential function 
$W(w_1,\ldots, w_6)$ of Figure \ref{example2} is 
\begin{align*}
W=\left\{-\li(\frac{w_3}{w_1})-\li(\frac{w_3}{w_5})+\li(\frac{w_2 w_3}{w_1 w_5})
+\li(\frac{w_1}{w_2})+\li(\frac{w_5}{w_2})-\frac{\pi^2}{6}+\log\frac{w_1}{w_2}\log\frac{w_5}{w_2}\right\}\\
+\left\{-\li(\frac{w_2}{w_1})-\li(\frac{w_2}{w_5})+\li(\frac{w_2 w_4}{w_1 w_5})
+\li(\frac{w_1}{w_4})+\li(\frac{w_5}{w_4})-\frac{\pi^2}{6}+\log\frac{w_1}{w_4}\log\frac{w_5}{w_4}\right\}\\
+\left\{-\li(\frac{w_4}{w_1})-\li(\frac{w_4}{w_5})+\li(\frac{w_3 w_4}{w_1 w_5})
+\li(\frac{w_1}{w_3})+\li(\frac{w_5}{w_3})-\frac{\pi^2}{6}+\log\frac{w_1}{w_3}\log\frac{w_5}{w_3}\right\}\\
+\left\{\li(\frac{w_1}{w_4})+\li(\frac{w_1}{w_6})-\li(\frac{w_1^2}{w_4 w_6})
-\li(\frac{w_4}{w_1})-\li(\frac{w_6}{w_1})+\frac{\pi^2}{6}-\log\frac{w_4}{w_1}\log\frac{w_6}{w_1}\right\}.
\end{align*}

Applying (\ref{main}), we obtain
\begin{eqnarray*}
w_1^{(0)}=\det(p,s_1)=5,~w_2^{(0)}=\det(p,s_2)=3,~w_3^{(0)}=\det(p,s_3)=7,\\
w_4^{(0)}=\det(p,s_4)=2,~w_5^{(0)}=\det(p,s_5)=1,~w_6^{(0)}=\det(p,s_6)=8,
\end{eqnarray*}
and $(w_1^{(0)},\ldots,w_6^{(0)})$ becomes a solution of 
$\mathcal{I}=\{\exp(w_k\frac{\partial W}{\partial w_k})=1~|~k=1,\ldots,6\}$.
Applying (\ref{W1}), we obtain
\begin{equation*}
W_0(w_1^{(0)},\ldots,w_6^{(0)})\equiv i(\vol(\rho)+i\,\cs(\rho))\modulo,
\end{equation*}
and numerical calculation verifies it by
\begin{equation*}
{W}_0(w_1^{(0)},\ldots,w_6^{(0)})=i(0+1.6449...i),
\end{equation*}
where $\vol(3_1)=0$ holds trivially and $1.6449...=\frac{\pi^2}{6}$ holds numerically.

\vspace{5mm}
\begin{ack}
  The author appreciates Seonhwa Kim. He already predicted the existence of a solution  
  when the author defined the potential function of the colored Jones polynomial several years ago.
  {This research was supported by Basic Science Research Program through the National Research Foundation of Korea (NRF)
  funded by the Ministry of Education (2014047764).}
\end{ack}

\bibliography{VolConj}
\bibliographystyle{abbrv}

{
\begin{flushleft}
{  Pohang Mathematics Institute (PMI),\\ Pohang 790-784, Republic of Korea\\}
  \vspace{0.4cm}
E-mail: dol0425@gmail.com\\
\end{flushleft}}
\end{document}